\documentclass[a4paper, pdftex]{amsart}

\usepackage[utf8]{inputenc}
\usepackage[T1]{fontenc}
\usepackage[pdftex]{graphicx}
\usepackage{amsmath}
\usepackage{amsfonts}
\usepackage{amsthm}
\usepackage{amssymb}
\usepackage{setspace}
\usepackage[format=hang]{caption}
\usepackage{thmtools}
\usepackage{thm-restate}
\usepackage{mathtools}
\usepackage{verbatim}
\usepackage[all, cmtip]{xy}
\usepackage{multirow}
\usepackage{bm}
\usepackage[textsize = tiny]{todonotes}
\usepackage{hyperref}
\hypersetup{colorlinks=true,linkcolor=blue,filecolor=blue,citecolor = 
blue,urlcolor=blue}
\usepackage[capitalize,nameinlink,noabbrev]{cleveref}
\usepackage{quiver}
\usepackage[figuresleft]{rotating}
\usepackage{pdflscape}
\usepackage{adjustbox}
\usepackage{xparse}

\usepackage[backend=biber, style=alphabetic, url=false, hyperref, 
backref, backrefstyle=none, maxbibnames=99, sorting=nyt]{biblatex}
\AtEveryBibitem{\clearlist{language}}
\RequirePackage{doi}

\usepackage{enumerate}

\newcommand{\VertexSet}[1]{\mathcal{V}_{#1}}

\newtheorem*{thm*}{Theorem}
\newtheorem{theorem}{Theorem}[section]
\newtheorem{corollary}[theorem]{Corollary}
\newtheorem{lemma}[theorem]{Lemma}

\newtheorem*{fact*}{Fact}
\newtheorem{proposition}[theorem]{Proposition}

\newcounter{theoremalph}

\theoremstyle{definition}

\newtheorem{definition}[theorem]{Definition}
\newtheorem{claim}[theorem]{Claim}

\theoremstyle{remark}
\newtheorem{remark}[theorem]{Remark}

\newtheorem*{notation}{Notation}

\renewcommand{\ll}{\left\langle}
\newcommand{\rr}{\right\rangle}
\newcommand{\ls}{\left\{}
\newcommand{\rs}{\right\}}

\newcommand{\mbQ}{\ensuremath{\mathbb{Q}}}
\newcommand{\mbZ}{\ensuremath{\mathbb{Z}}}
\newcommand{\mbN}{\ensuremath{\mathbb{N}}}
	
\newcommand{\on}[1]{\operatorname{#1}}

\newcommand{\overbar}[1]{\mkern 2mu\overline{\mkern-2mu#1\mkern-2mu}\mkern 2mu}

\newcommand{\St}{\operatorname{St}}

\newcommand{\im}{\operatorname{im}}

\newcommand{\SL}[2]{\ensuremath{\operatorname{SL}_{#1}(#2)}}
\newcommand{\Sp}[2]{\ensuremath{\operatorname{Sp}_{#1}(#2)}}

\newcommand{\I}[1][n]{\mathcal{I}_{#1}}
\NewDocumentCommand {\Irel} { O{n} O{m} }{\mathcal{I}^{#2}_{#1}}
\newcommand{\Idel}[1][n]{\mathcal{I}_{#1}^{\delta}}
\NewDocumentCommand {\Idelrel} { O{n} O{m} }{\mathcal{I}^{\delta,#2}_{#1}}
\newcommand{\Isigdel}[1][n]{\mathcal{I}_{#1}^{\sigma,\delta}}
\NewDocumentCommand {\Isigdelrel} { O{n} O{m} }{\mathcal{I}^{\sigma,\delta,#2}_{#1}}
\newcommand{\IA}[1][n]{\mathcal{IA}_{#1}}
\NewDocumentCommand {\IArel} { O{n} O{m} }{\mathcal{IA}^{#2}_{#1}}

\NewDocumentCommand {\IAArel} { O{n} O{m} }{\mathcal{IAA}^{#2}_{#1}}

\NewDocumentCommand {\IAAstrel} { O{n} O{m} }{\mathcal{IAA}^{*,#2}_{#1}}

\providecommand{\Link}{\ensuremath\mathsf{Link}}

\providecommand{\Star}{\ensuremath\mathsf{Star}}

\providecommand{\mbZ}{\ensuremath\mathbb Z}

\newcommand{\efspan}{\langle \vec e_{n}, \vec f_{n} \rangle_\mbQ}

\newcommand{\sym}[1]{[\![ #1 ]\!]}

\newcommand{\vcd}{\on{vcd}}

\author{Benjamin Br\"uck} \thanks{BB was supported by the Danish National Research Foundation through the Copenhagen Centre for Geometry and Topology (DNRF151) and by the Deutsche Forschungsgemeinschaft (DFG, German Research Foundation) under Project-ID 427320536 -- SFB 1442, as well as by Germany's Excellence Strategy EXC 2044 – 390685587, Mathematics Münster: Dynamics–Geometry–Structure.}
\address{Institut für Mathematische Logik und Grundlagenforschung, University of Münster, Germany}
\email{benjamin.brueck@uni-muenster.de}

\author{Robin J. Sroka} \thanks{RJS was supported by the European Research Council (ERC grant agreement No.772960) and the Danish National Research Foundation (DNRF92, DNRF151) as a PhD student at the University of Copenhagen, by NSERC Discovery Grant A4000 in connection with a Postdoctoral Fellowship at McMaster University, and by the Deutsche Forschungsgemeinschaft (DFG, German Research Foundation) -- Project-ID 427320536 -- SFB 1442, as well as by Germany’s Excellence Strategy EXC 2044 390685587, Mathematics Münster: Dynamics–Geometry–Structure as a Postdoctoral Research Associate at the University of Münster.}
\address{Department of Mathematics \& Statistics, McMaster University, 
Hamilton, Canada}
\curraddr{Mathematisches Institut, Universität Münster, Germany}
\email{robinjsroka@uni-muenster.de}

\title{Apartment classes of integral symplectic groups}

\begin{document}
	
\begin{abstract}
  In this note we present an alternative proof of a theorem of Gunnells, which states that the Steinberg module of $\Sp{2n}{\mbQ}$ is a cyclic $\Sp{2n}{\mbZ}$-module, generated by integral apartment classes.
\end{abstract}

\maketitle

\section{Introduction}

Consider the vector space $\mbQ^{2n}$ equipped with the standard symplectic 
form $\omega$, i.e.\ the skew-symmetric, non-degenerate bilinear form which on 
the standard basis $\{\vec e_1, \vec f_1,  \dots, \vec e_{n}, \vec f_{n}\}$ 
evaluates as
\begin{align*}
	\omega(\vec e_i, \vec e_j) = \omega(\vec f_i, \vec f_j) = 0 & \text{ for } i,j \in \{1, \dots, n\},\\
	\omega(\vec e_i, \vec f_j) = 0 & \text{ for } i \neq j \in \{1, \dots, n\},\\
	\omega(\vec e_i, \vec f_i) = - \omega(\vec f_i, \vec e_i) = 1 & \text{ for } i \in \{1, \dots, n\}.
\end{align*}
The symplectic group $\Sp{2n}{\mbQ}$ is the group of $\mbQ$-linear 
automorphisms of $\mbQ^{2n}$ that preserve $\omega$. 
Restricting to $\mbZ^{2n}$, we obtain the symplectic module $(\mbZ^{2n}, 
\omega)$ and the integral symplectic group $\Sp{2n}{\mbZ}$. We may use the 
standard basis to identify $\Sp{2n}{\mbQ}$ and $\Sp{2n}{\mbZ}$ with $2n \times 
2n$-matrix groups.

This work concerns \emph{the symplectic Steinberg module} $\St^\omega_n$, an 
important $\Sp{2n}{\mbZ}$-representation (the rational dualizing module of $\Sp{2n}{\mbZ}$ \cite{borelserre1973cornersandarithmeticgroups}) that can be constructed as follows: A subspace $V \subseteq \mbQ^{2n}$ is called isotropic if $\omega|_V$ is zero. The Tits building of type $\mathtt{C}_n$ over 
$\mbQ$ is the poset $T^{\omega}_n$ of all nontrivial proper isotropic subspaces 
$V$ of $(\mbQ^{2n}, \omega)$ ordered by the inclusion of subspaces. This poset 
admits a natural $\Sp{2n}{\mbZ}$-action because symplectic 
matrices map isotropic subspaces to isotropic subspaces. A theorem of 
Solomon--Tits \cite{solomon1969thesteinbergcharacterofafinitegroup} 
implies that $T^\omega_n \simeq \bigvee S^{n-1}$ has the homotopy type of a 
bouquet of $(n-1)$-spheres. The Steinberg module of $\Sp{2n}{\mbQ}$ is the 
$\Sp{2n}{\mbZ}$-module that arises as the reduced top-degree homology of the 
symplectic Tits building, 
\[
\St^\omega_n \coloneqq \tilde{H}_{n-1}(T^\omega_n; \mbZ).
\]

The goal of this work is to give an alternative proof of a theorem of 
Gunnells \cite[4.11.\ Theorem]{gunnells2000symplecticmodularsymbols}, which 
shows that $\St^\omega_n$ is a cyclic $\Sp{2n}{\mbZ}$-module and 
describes an 
explicit set of generators for $\St^\omega_n$.

\begin{theorem}[Gunnells] \label{gunnellstheorem}
	There exists an $\Sp{2n}{\mbZ}$-equivariant surjection
	$$[-]: \mbZ[\Sp{2n}{\mbZ}] \twoheadrightarrow \St^\omega_n$$
	for all $n \geq 1$ such that the set of generators of $\St^\omega_n$ given by $\{[M]: M \in \Sp{2n}{\mbZ}\}$ is equal to the set of integral apartment classes (see \cref{modularsymbolmap}).
\end{theorem}

\cref{gunnellstheorem} is the special case $\mathcal{O} = \mbZ$ of Gunnells' 
result \cite[4.11 Theorem]{gunnells2000symplecticmodularsymbols}, which allows 
any Euclidean ring of integers $\mathcal{O}$ of a number field $K/\mbQ$.

Our primary interest in $\St^\omega_n$ stems from applications 
to the study of the rational cohomology of $\Sp{2n}{\mbZ}$: Using ideas contained in \cite{borelserre1973cornersandarithmeticgroups, churchfarbputman2019integralityinthesteinbergmodule}, Gunnells' 
generating set for $\St^\omega_n$ can be used to prove that the rational 
cohomology of $\Sp{2n}{\mbZ}$ vanishes in its virtual cohomological 
dimension $\vcd_n = n^2$, $$H^{n^2}(\Sp{2n}{\mbZ}; \mbQ) = 0$$
if $n \geq 1$ (see Brück--Patzt--Sroka \cite[Chapter 
5]{sroka2021thesis} or, for a more general version, Brück--Santos Rego--Sroka 
\cite{bruecksantosregosroka2022onthetopdimensionalcohomologyofarithmeticchevalleygroups}).
 In a sequel that is joint work with Peter Patzt \cite{brueckpatztsroka2023inpreparation}, we build on the techniques developed in this note to determine the relations between all integral apartment classes, i.e.\ between the generators of $\St^\omega_n$ appearing in \cref{gunnellstheorem}. This yields a presentation of the symplectic Steinberg module $\St^\omega_n$ for $n \geq 1$. We use this presentation to prove that the rational cohomology of $\Sp{2n}{\mbZ}$ vanishes 
one degree below its virtual cohomological dimension,
$$H^{n^2-1}(\Sp{2n}{\mbZ}; \mbQ) = 0$$
if $n \geq 2$ (see Brück--Patzt--Sroka \cite{brueckpatztsroka2023inpreparation} for $n \geq 3$; and Igusa \cite{igusa1962}, Hain \cite{hain2002} and Hulek-Tommasi \cite{ht2012} for $n \leq 4$).

Our motivation for writing the present note is threefold: Gunnells' 
algorithmic strategy of proof for  \cref{gunnellstheorem} in 
\cite{gunnells2000symplecticmodularsymbols} was inspired by work of 
Ash--Rudolph for special linear groups 
\cite{ashrudolph1979themodularsymbolandcontinuedfractionsinhigherdimensions}. 
Our first aim here is to implement an idea of Putman 
\cite{putman2021privatecommunication} and 
present a new, more geometric argument in the style of recent work 
on special linear groups  
\cite{churchfarbputman2019integralityinthesteinbergmodule, 
churchputman2017thecodimensiononecohomologyofslnz}. In fact, our 
strategy to study integral symplectic groups relies on and uses results 
obtained in \cite{churchfarbputman2019integralityinthesteinbergmodule, 
churchputman2017thecodimensiononecohomologyofslnz}. This idea is also prominent 
in the sequel \cite{brueckpatztsroka2023inpreparation}, which relies on results 
contained in \cite{churchfarbputman2019integralityinthesteinbergmodule, 
churchputman2017thecodimensiononecohomologyofslnz, 
brueckmillerpatztsrokawilson2022onthecodimensiontwocohomologyofslnz}. Our 
second aim is to develop and showcase some of the techniques used in 
\cite{brueckpatztsroka2023inpreparation} at a simpler example. Our approach to 
\cref{gunnellstheorem} requires to 
show that a certain simplicial complex $\IA$ is highly connected. The difficult 
parts of this connectivity calculation have been carried out by Putman in \cite{putman2009aninfinitepresentationofthetorelligroup}. However, Putman 
informed us that \cite{putman2009aninfinitepresentationofthetorelligroup} 
contains small gaps. Our third aim in this article is to explain how 
these can be filled.
 
 \textbf{Outline.} \cref{sec:restricted-building} introduces a new poset, the restricted Tits building 
 $T^\omega_n(W)$, and studies its 
 connectivity properties. This poset is a variant of the symplectic Tits building 
 $T^\omega_n$ defined above. In \cref{sec:putmans-results-and-IArel}, we 
 define the complex $\I^{\sigma, \delta}$, which Putman introduced in 
 \cite{putman2009aninfinitepresentationofthetorelligroup}, and the complex 
 $\IA$, which we use in our proof of Gunnells' theorem. We show that Putman's 
 connectivity results for $\I^{\sigma, \delta}$ imply that $\IA$ is highly 
 connected as well. Furthermore, we explain how one can combine connectivity 
 results obtained by Church--Putman 
 \cite{churchputman2017thecodimensiononecohomologyofslnz} with our results 
 about the restricted Tits building to give an alternative proof of the first 
 steps of Putman's connectivity calculation for $\I^{\sigma, \delta}$. This 
 fixes the gaps in Putman's argument (see \cref{rem:smallgap} and 
 \cref{rem:revisited}). In 
 \cref{sec:symplectic-modular-symbols}, we define the integral apartment 
 class map appearing in Gunnells' theorem. 
 \cref{sec:new-proof-of-gunnells-theorem} 
 contains the new proof of \cref{gunnellstheorem}.

\textbf{Acknowledgments.}
This article is based on Chapter 5 of Sroka's PhD Thesis 
\cite{sroka2021thesis} written at the University of Copenhagen. Essentially,
all results presented here are contained in \cite{sroka2021thesis}. It is a 
pleasure to thank Andrew Putman for posing the question that led to this work 
\cite{amssession2020} and for sharing his idea for a new proof of 
Gunnells’ theorem \cite{putman2021privatecommunication}. We thank Peter 
Patzt for helpful discussions and comments, and the department of the 
University of Copenhagen for the excellent working conditions. RJS would 
like to thank his PhD advisor Nathalie Wahl for many fruitful and clarifying 
conversations about \cite[Chapter 5]{sroka2021thesis}.
We thank the anonymous referee for their careful reading and helpful suggestions.

\section{The restricted Tits building}
\label{sec:restricted-building}

This section introduces and studies a new poset, the restricted Tits building 
$T^\omega_n(W)$. In the next section, we 
use this poset to fix gaps in an argument contained in 
\cite{putman2009aninfinitepresentationofthetorelligroup}. The results for 
$T^\omega_n(W)$ presented here are also used in the sequel to this work 
\cite{brueckpatztsroka2023inpreparation}. We assume that $n \geq 1$ throughout this section. 

Recall that the symplectic Tits building $T_n^\omega$ is the 
poset of nontrivial isotropic subspaces of $\mbQ^{2n}$ ordered by inclusion of 
subspaces. The order complex of this poset, which we also denote by 
$T_n^\omega(\mbQ)$, is an ordered simplicial complex with $k$-simplices given 
by the following set of flags
\[
\{ V_0 \subsetneq \dots \subsetneq V_k : 0 \neq V_i \subsetneq \mbQ^{2n} \text{ 
isotropic subspace} \}.
\]
This complex has dimension $n-1$ and the $i$-th face of a $k$-simplex is 
obtained by omitting the $i$-th isotropic subspace $V_i$ of the flag.

\begin{definition}
We define
\begin{equation*}
	W = \ll \vec e_1, \vec f_1 \ldots, \vec e_{n-1}, \vec f_{n-1}, 
\vec e_{n} \rr_{\mbQ} \subseteq \mbQ^{2n}
\end{equation*}
to be the subspace of $\mbQ^{2n}$ spanned by all standard basis vectors apart from $f_n$.

	We denote by $T^{\omega}_n(W)$ the subposet of the symplectic Tits building $T^{\omega}_n$ consisting of isotropic subspaces $V \in T^\omega_n$ that are contained in $W$.
\end{definition}

Recall from \cite[p.116-117]{quillen1978homotopypropertiesofposets} that a 
poset $P$ is Cohen--Macaulay of dimension $d$ if the associated order complex 
$P$ is $d$-spherical, i.e.~$d$-dimensional and homotopy equivalent to a wedge of $d$-spheres, and the link of each $k$-simplex in $P$ is 
$(d-k-1)$-spherical. The main result of this section is the following theorem.

\begin{theorem} \label{restrictedtitsbuildingcm}
  $T^{\omega}_n(W)$ is a contractible Cohen--Macaulay poset of dimension $n-1$.
\end{theorem}

For any subspace $H \subseteq \mbQ^{2n}$, let $$H^\perp = 
\{\vec v \in \mbQ^{2n}: \omega(\vec v, \vec h) = 0 \text{ for all } \vec h \in 
H\}$$ denote the symplectic complement of $H$ in $\mbQ^{2n}$. The following two 
observations are the main ingredients of the proof of 
\cref{restrictedtitsbuildingcm}.

\begin{lemma} \label{adding}
  If $V \in T^{\omega}_n(W)$, then $\ll \vec e_{n} \rr_{\mbQ} + V \in T^{\omega}_n(W)$.
\end{lemma}

\begin{proof}
	Observe that $W \subseteq \ll \vec e_{n} \rr_{\mbQ}^\perp$, hence $V \subseteq \ll \vec e_{n} \rr_{\mbQ}^\perp$. It follows that $\ll \vec e_{n} \rr_{\mbQ} + V \subseteq W$ is isotropic. Hence, $\ll \vec e_{n} \rr_{\mbQ} + V \in T^{\omega}_n(W)$.
\end{proof}

Given a poset $P$, the upper link $P_{>x}$ and the lower link $P_{<x}$ of an 
element $x \in P$ are the subposets of $P$ containing all elements $y \in P$ 
that satisfy $y > x$ and $y < x$ respectively. The interval $(x,z)$ between two 
elements $x \leq z \in P$ is the subposet of $P$ consisting of all $y \in P$ 
that satisfy $x < y < z$. The next observation concerns certain upper links in 
$T^{\omega}_n(W)$ and is similar to Lemma 4.2 of Sprehn--Wahl 
\cite{sprehnwahl2020homologicalstabilityforclassicalgroups}.

\begin{lemma} \label{intersecting}
  If $\ll \vec e_{n} \rr_{\mbQ} \subseteq Q \in T^{\omega}_n(W)$, then 
  \[
  T^{\omega}_n(W)_{> Q} \cong T^{\omega}(\efspan^\perp)_{> Q \cap \efspan^\perp 
  }.
  \]
  For the case $Q = \ll \vec e_{n} \rr_{\mbQ}$, we set 
  $T^{\omega}(\efspan^\perp)_{> Q \cap \efspan^\perp} \coloneqq 
  T^{\omega}(\efspan^\perp)$.
\end{lemma}

\begin{proof}
	Note that any $V \in T^{\omega}_n(W)_{> Q}$ admits a direct sum decomposition $V = \ll \vec e_{n} \rr_{\mbQ} \oplus (V \cap \efspan^\perp)$. The poset maps $$T^{\omega}_n(W)_{> Q} \to T^\omega(\efspan^\perp)_{> Q \cap \efspan^\perp}: V \mapsto V \cap \efspan^\perp$$ and $$T^{\omega}(\efspan^\perp)_{> Q \cap \efspan^\perp} \to T^{\omega}_n(W)_{> Q}: V \mapsto \ll \vec e_{n} \rr_{\mbQ} \oplus V$$ are therefore inverses of each other.
\end{proof}

\begin{lemma} \label{contractible}
  $T^{\omega}_n(W)$ is contractible.
\end{lemma}

\begin{proof}
	The poset map $$f: T^{\omega}_n(W) \to T^{\omega}_n(W): V \mapsto \ll \vec 
	e_{n} \rr_{\mbQ} + V$$ is well-defined by \autoref{adding} and satisfies $V \subseteq f(V)$. It follows 
	from \cite[§1.5]{quillen1978homotopypropertiesofposets} that 
	$T^{\omega}_n(W)$ is homotopy equivalent to $\im(f)$ and that $\im(f)$ is 
	contractible using the cone point $\ll \vec e_{n} \rr_{\mbQ}$.
\end{proof}

\begin{lemma}\label{lem:restrictedtitsbuildingcmproperty}
  $T^{\omega}_n(W)$ is a Cohen--Macaulay poset of dimension $n-1$.
\end{lemma}

\begin{proof}
	This proof uses the characterization of Cohen--Macaulay 
	posets given in \cite[Proposition 
	8.6.]{quillen1978homotopypropertiesofposets}: Let $Q' \subseteq Q \in 
	T^{\omega}_n(W)$. We need to see that the lower link $T^{\omega}_n(W)_{< 
	Q}$ is $(\dim Q - 
	2)$-spherical, the interval $(Q', Q)$ is $(\dim Q - \dim Q' -2)$-spherical 
	and the upper link $T^{\omega}_n(W)_{> Q}$ is $(n - \dim Q - 1)$-spherical.

	Connectivity of the lower link and the interval: Note that 
	$T^{\omega}_n(W)_{< Q}$ 
	is the poset of nontrivial proper subspaces of $Q$. This is exactly a 
	Tits building $T(Q)$ of type $\mathtt{A}_{\dim(Q)-1}$, which is known 
	to be a Cohen--Macaulay poset of dimension $(\dim Q - 
	2)$ (see \cite{solomon1969thesteinbergcharacterofafinitegroup} and 
	\cite[IV.5 Remark 2]{brown1989buildings}). Therefore, $T^{\omega}_n(W)_{< 
	Q}$ is $(\dim Q - 2)$-spherical and $(Q', Q)$ is $((\dim Q -2) - (\dim Q' - 
	1) - 1) = (\dim Q - \dim Q' - 2)$-spherical.

	Connectivity of the upper link: We consider two cases.

	\begin{enumerate}
	\item \label{it_first_part_upper_link_proof}Assume that $\ll \vec e_{n} \rr_{\mbQ} \not\subseteq Q$. Then $\ll 
	\vec e_{n} \rr_{\mbQ} + Q \in T^{\omega}_n(W)_{> Q}$ is a cone point of the 
	image of the monotone poset map $f: V \mapsto \ll \vec e_{n} \rr_{\mbQ} + 
	V$ on $T^{\omega}_n(W)_{> Q}$. It follows from 
	\cite[§1.5]{quillen1978homotopypropertiesofposets} that $T^{\omega}_n(W)_{> 
	Q}$ is contractible and in particular $(n - \dim Q - 1)$-spherical.
	\item Assume that $\ll \vec e_{n} \rr_{\mbQ} \subseteq Q$. Then \autoref{intersecting} yields the identification $$T^{\omega}_n(W)_{> Q} \cong T^{\omega}(\efspan^\perp)_{> Q \cap \efspan^\perp}.$$
	But $T^{\omega}(\efspan^\perp)$ is Cohen--Macaulay of dimension $(n-2)$ (see \cite{solomon1969thesteinbergcharacterofafinitegroup} and \cite[IV.5 Remark 2]{brown1989buildings}). Therefore, $T^{\omega}_n(W)_{> Q \cap \efspan^\perp}$ is spherical of dimension
	\begin{equation*}
	(n-2) - (\dim (Q \cap \efspan^\perp)-1) - 1 = n - \dim Q - 1.\qedhere
\end{equation*}
	\end{enumerate}
\end{proof}

\section{Putman's connectivity results revisited and the complex
\texorpdfstring{$\IArel$}{IA}}
\label{sec:putmans-results-and-IArel}

This section introduces the simplicial complex $\IA$ that plays a key role 
in our proof of Gunnells' \cref{gunnellstheorem}. The simplicial complex $\IA$ 
contains a subcomplex $\I^{\sigma, \delta} \hookrightarrow \IA$, which has been 
studied by Putman in \cite{putman2009aninfinitepresentationofthetorelligroup}. 
The goal of this section is twofold. After defining simplicial complexes 
related to $\I^{\sigma, \delta}$ and $\IA$, we outline the strategy that 
Putman used in \cite{putman2009aninfinitepresentationofthetorelligroup} to 
prove that $\I^{\sigma, \delta}$ is highly connected. Our first goal is to give 
alternative proofs for the first steps of Putman's argument, filling gaps in \cite{putman2009aninfinitepresentationofthetorelligroup} 
(see \cref{rem:smallgap} and \cref{rem:revisited}). For this, we 
combine the results for the restricted Tits building $T^{\omega}_n(W)$ obtained 
in the last section with connectivity calculations of Church--Putman 
\cite{churchputman2017thecodimensiononecohomologyofslnz}. Our second goal is to 
show that the complex $\IA$ can be constructed from $\I^{\sigma, \delta}$ by 
attaching simplices along highly connected links. As a consequence, Putman's 
connectivity result for $\I^{\sigma, \delta}$ implies that $\IA$ is highly 
connected as well. The high-connectivity of $\IA$ and the link structure of 
$\IA$ are exactly the properties that make the induction argument in the proof 
of Gunnells' theorem in \cref{sec:new-proof-of-gunnells-theorem} work.

\begin{notation}
	Throughout this subsection, we assume that $m,n \in \mbN$ denote natural numbers satisfying $m+n \geq 1$. We 
	consider $\mbZ^{2(m+n)} \subset \mbQ^{2(m+n)}$ equipped with the standard 
	symplectic form $\omega$ and denote its standard symplectic basis by 
	$\{\vec e_1,\vec f_1, \dots , \vec e_{m+n}, \vec f_{m+n}\}$. Given a 
	primitive vector $\vec v \in \mbZ^{2(m+n)}$, we write $v = \ll \vec v 
	\rr_{\mbZ}$ for the rank-1 summand it spans in $\mbZ^{2(m+n)}$. Similarly, 
	given a rank-1 
	summand $v$ of $\mbZ^{2(m+n)}$, we write $\vec v$ for some choice of primitive vector in $v$. Note that there are exactly two such choices, the other one being $- \vec v$.
\end{notation}

\begin{definition} \label{typesofsimplices} 
	Let $\VertexSet{m+n}$ be the set 
	\[
	\VertexSet{m+n} \coloneqq \{v \subseteq \mbZ^{2(m+n)}: v \text{ is a rank-1 summand of } 
	\mbZ^{2(m+n)}\}.
	\]
	A subset $\Delta = \{v_0, \dots, v_k\} 
	\subset \VertexSet{m+n}$ of $k+1$ lines in $\mbZ^{2(m+n)}$ is called
	\begin{itemize}
		\item a \emph{standard} simplex if $\langle \vec v_i :0 \leq i \leq k \rangle_{\mbZ}$ is an isotropic rank-$(k+1)$ summand of $\mbZ^{2(m+n)}$;
		\item a \emph{2-additive} simplex if $\vec v_0 = \pm \vec v_1 \pm \vec v_2$ and $\Delta \setminus \{v_0\}$ is a standard $(k-1)$-simplex;
		\item a \emph{$\sigma$} simplex if $\omega(\vec v_k, \vec v_{k-1}) = \pm 1$, $\omega(\vec v_k, \vec v_i) = 0$ for $0 \leq i \leq k-2$ and $\Delta \setminus \{v_k\}$ is a standard $(k-1)$-simplex;
		\item a \emph{mixed} simplex if $\Delta \setminus \{v_0\}$ is a $\sigma$ simplex, $\Delta \setminus \{v_k\}$ is a 2-additive simplex and $\omega(v_0, v_k) = 0$.
	\end{itemize}
\end{definition}

\begin{definition}
  The simplicial complexes $\I[m+n], \I[m+n]^\delta,\I[m+n]^{\sigma, \delta}$ and $\IA[m+n]$ have $\VertexSet{m+n}$ as their vertex set and
  \begin{itemize}
  \item the simplices of $\I[m+n]$ are all standard;
  \item the simplices of $\I[m+n]^\delta$ are all either standard or 2-additive;
  \item the simplices of $\I[m+n]^{\sigma, \delta}$ are all either standard, 2-additive or $\sigma$;
  \item the simplices of $\IA[m+n]$ are all either standard, 2-additive, $\sigma$ or mixed.
  \end{itemize}
\end{definition}

\begin{definition}
	\label{def_linkhat}
	Let $X_{m+n}$ denote the complex $\I[m+n], \Idel[m+n], \Isigdel[m+n]$ or 
	$\IA[m+n]$. We define $X_n^m$ to be the full subcomplex of 
	$\Link_{X_{m+n}}(\ls e_1, \ldots, e_m \rs)$ on the vertex set of lines $v 
	\in \Link_{X_{m+n}}(\ls e_1, \ldots, e_m \rs)$ satisfying the 
	following:
	\begin{enumerate}
		\item $\vec v \notin \langle \vec e_1, \dots, \vec e_m \rangle_{\mbZ}$.
		\item For $1 \leq i \leq m$, we have $\omega(\vec e_i , \vec v) = 0$, 
		i.e.~there are no $\sigma$ edges between $v$ and the vertices of $\ls 
		e_1, \ldots, e_m \rs$.
	\end{enumerate}
\end{definition}

\begin{definition}
	\label{def:idelrelw}
  Let $W = \langle \vec e_1, \vec f_1, \dots, \vec e_{m+n-1}, \vec f_{m+n-1}, \vec e_{m+n} \rangle_\mbQ \subseteq \mbQ^{2(m+n)}$ be as in 
  \cref{sec:restricted-building}. We write $\Irel(W)$ and $\Idelrel(W)$ for the 
  full subcomplexes of $\Irel$ and $\Idelrel$, respectively, on the set of 
  vertices contained in $W$.
\end{definition}

The complexes $\Irel(W), \Irel, \Idelrel, \Idelrel(W)$ and $\Isigdelrel$ have 
been defined and studied by Putman 
\cite[Section 6]{putman2009aninfinitepresentationofthetorelligroup}. The next theorem 
lists the five steps of Putman's proof that $\Isigdelrel$ is spherical.

\begin{theorem}[Putman, 
{\cite[Proposition 6.13 and 
6.11]{putman2009aninfinitepresentationofthetorelligroup}}] 
\label{isigmadelta-connectivity}
  Let $m \geq 0$ and $n \geq 1$, then:
  \begin{enumerate}
  \item \label{it_conn_IrelW} $\Irel(W)$ is $(n-2)$-connected.
  \item \label{it_conn_Irel} $\Irel$ is $(n-2)$-connected.
  \item \label{it_conn_IdelrelW} $\Idelrel(W)$ is $(n-1)$-connected.
  \item \label{it_conn_incl}$\Idelrel \hookrightarrow \Isigdelrel$ is the zero 
  map on $\pi_k$ for $0 \leq k \leq n-1$.
  \item \label{it_conn_isigdelrel} $\Isigdelrel$ is $n$-dimensional and 
  $(n-1)$-connected.\footnote{\cite[Theorem 6.11]{putman2009aninfinitepresentationofthetorelligroup} states this result only for $m = 0$. Its proof uses \cref{it_conn_Irel} and \cref{it_conn_incl} for $m = 0$. For $m > 0$, the same argument works if one uses \cref{it_conn_Irel} and \cref{it_conn_incl} for $m > 0$.}
  \end{enumerate}
\end{theorem}

Putman made us aware that the proof of 
\cite[Proposition 6.13]{putman2009aninfinitepresentationofthetorelligroup} 
contains some small gaps. They occur in the
proof of \cref{it_conn_IrelW} and \cref{it_conn_IdelrelW} 
of \cref{isigmadelta-connectivity}, and are explained in the next remark.

\begin{remark}
\label{rem:smallgap}
\cite[Proofs of Proposition 6.13.1 and 
6.13.3]{putman2009aninfinitepresentationofthetorelligroup} assert -- 
without proof -- that certain isomorphisms of simplicial complexes exist, 
but it is seems unclear why this would be the case. Using the notation of  
\cite{putman2009aninfinitepresentationofthetorelligroup}, the claims are as follows.
\begin{enumerate}
	\item \cite[Page 632. Proof of Proposition 6.13, first conclusion, third 
	paragraph.]{putman2009aninfinitepresentationofthetorelligroup} asserts that 
	$\on{link}_{\mathcal{L}^{\Delta^k, W}(g)}(\phi(\Delta')) \cong 
	\mathcal{L}^{k+m', W}(g)$.
	\item \cite[Page 634. Proof of Proposition 6.13, third conclusion. Step 1.]{putman2009aninfinitepresentationofthetorelligroup} asserts that 
	$\on{link}_{\mathcal{L}_{\delta}^{\Delta^k, W}(g)}(\phi(t)) \cong 
	\mathcal{L}^{k+(m'-1), W}(g)$. \cite[Page 634. Proof of Proposition 
	6.13, third conclusion. Step 2, Case 
	2.]{putman2009aninfinitepresentationofthetorelligroup} asserts an  
	isomorphism
	$\mathcal{L}^{\Delta^k \cup \{\langle v \rangle\}, W}(g) \cong 
	\mathcal{L}^{k+1, W}(g)$. \cite[Page 634. Proof of Proposition 6.13, third 
	conclusion. Step 3.]{putman2009aninfinitepresentationofthetorelligroup} 
	asserts that
	$\mathcal{L}^{\Delta^k \cup \{\phi(x)\}, W}_\delta(g) \cong 
	\mathcal{L}^{k+1, W}_\delta(g)$.
\end{enumerate}

We use the first claim to illustrate why these 
assertions are difficult to verify. 
Translated to the notation of the present note, it is as follows:
\begin{enumerate}
\item Let $0 \leq k \leq n-2$ and $\Delta$ a $k$-simplex  in 
$\Irel(W)$. Then 
\[
\Link_{\Irel(W)}(\Delta) \cong \Irel[n-k-1][m+k+1](W).
\]
\end{enumerate}
Let 
$m = 0$, $n \geq 2$ and consider the $0$-simplex $\Delta = e_{n} \in \I(W)$. Then the claim asserts that
\begin{equation*}
	\Link_{\I(W)}(e_{n}) \cong \Link_{\I(W)}(e_{1}).
\end{equation*} 
For $\I$, i.e.\ if the vertex set has \emph{not} been restricted using $W$, it is indeed easy to see that there is an isomorphims $\Link_{\I}(e_{n}) \cong \Link_{\I}(e_{1})$.
However, while there is an equality $\Link_{\I(W)}(e_{n}) = \Link_{\I}(e_{n})$, the inclusion $\Link_{\I(W)}(e_{1}) \subsetneq \Link_{\I}(e_{1})$ is strict because e.g.\ the vertex $f_n$ is not contained in $\Link_{\I(W)}(e_1)$.
Therefore the first assertion about links in $\I(W)$ has the following consequence:
\begin{equation*}
	\Link_{\I}(e_{1}) \cong \Link_{\I}(e_{n}) = \Link_{\I(W)}(e_{n})   \cong \Link_{\I(W)}(e_{1}) \subsetneq \Link_{\I}(e_{1}),
\end{equation*} 
i.e.\ it identifies $\Link_{\I}(e_1)$ with a proper subcomplex of itself. It is not clear why such an identification would exists. Similar issues arise if one tries to verify the other claims.
\end{remark}

\begin{remark}
\label{rem:revisited}
	The proofs of \cite[Proposition 6.13.2 and 6.13.4]{putman2009aninfinitepresentationofthetorelligroup}, i.e.\ 
	\cref{it_conn_Irel} and \cref{it_conn_incl} of 
	\cref{isigmadelta-connectivity}, use
	identifications which are analogous to the assertions 
	described in \cref{rem:smallgap}. The important difference is that in 
	\cite[Proposition 6.13.2 and 6.13.4]{putman2009aninfinitepresentationofthetorelligroup} these are 
	identifications of complexes whose vertex set has \emph{not} been 
	restricted using $W_{m+n} = \langle \vec e_1, \vec f_1, \dots, \vec 
	e_{m+n-1}, \vec f_{m+n-1}, \vec e_{m+n} \rangle_\mbQ \subseteq 
	\mbQ^{2(m+n)}$. Hence, these identifications can easily be verified. 
	The proofs of \cite[Proposition 6.13.2 and 
	6.13.4]{putman2009aninfinitepresentationofthetorelligroup} furthermore rely on 
	\cite[Proposition 6.13.1 and 
	6.13.3]{putman2009aninfinitepresentationofthetorelligroup}. Since we provide 
	alternative arguments for these two statements, the proofs of 
	\cite[Proposition 
	6.13.2 and 6.13.4]{putman2009aninfinitepresentationofthetorelligroup} are 
	unaffected by the discussion in \cref{rem:smallgap}. Similarly, the proof 
	of \cite[Proposition 	
	6.11]{putman2009aninfinitepresentationofthetorelligroup}, i.e.\ 
	\cref{it_conn_isigdelrel} in \cref{isigmadelta-connectivity}, is 
	not affected.
\end{remark}

We now start working towards an alternative proof of \cite[6.13.1 and 6.13.3]{putman2009aninfinitepresentationofthetorelligroup}, i.e.\ 
\cref{it_conn_IrelW} and \cref{it_conn_IdelrelW} of \cref{isigmadelta-connectivity}, using the restricted 
Tits building introduced in the previous subsection and connectivity 
calculations obtained by Church--Putman 
\cite{churchputman2017thecodimensiononecohomologyofslnz}. This fixes the gaps in
\cite{putman2009aninfinitepresentationofthetorelligroup} 
outlined in \cref{rem:smallgap}.

\begin{definition}
Let $W_{m+n} = \langle \vec e_1, \vec f_1, \dots, \vec e_{m+n-1}, \vec f_{m+n-1}, \vec e_{m+n} \rangle_\mbQ \subseteq \mbQ^{2(m+n)}$
be as in \cref{sec:restricted-building}.
We write
\begin{equation*}
T^{\omega,m}_n \coloneqq (T^{\omega}_{m+n})_{> \ll \vec e_1, \dots, \vec e_m 
\rr_{\mbQ}} 
\text{ and }
T^{\omega, m}_n(W) \coloneqq T(W_{m+n})_{> \ll \vec e_1, \dots, \vec e_m \rr_\mbQ}
\end{equation*}
for the upper links of the isotropic subspace $\ll \vec e_1, \dots, \vec e_m \rr_{\mbQ}$ in $T^\omega_{m+n}$ and  $T^\omega_{m+n}(W)$, respectively.
\end{definition}

\begin{lemma}
	\label{cor:properties-of-restricted-tits-building}
	Let $m \geq 0$ and $n \geq 1$. $T^{\omega,m}_n$ and $T^{\omega, m}_n(W)$ are Cohen--Macaulay posets of dimension $(n-1)$. Furthermore, $T^{\omega, m}_n(W)$ is contractible.
\end{lemma}

\begin{proof}
	For the upper links in the symplectic Tits building $T^{\omega,m}_n$, 
	this follows from the Solomon--Tits Theorem (see 
	\cite{solomon1969thesteinbergcharacterofafinitegroup} and \cite[IV.5 Remark 
	2]{brown1989buildings}) and \cite[Proposition 
	8.6]{quillen1978homotopypropertiesofposets}. For the upper links in the 
	restricted Tits building $T^{\omega,m}_n(W)$, it follows from 
	\cref{restrictedtitsbuildingcm} and \cite[Proposition 
	8.6]{quillen1978homotopypropertiesofposets}. The contractibility of 
	$T^{\omega, m}_n(W) = T^{\omega}_{m+n}(W)_{> \ll \vec e_1, \dots, \vec e_m 
	\rr_\mbQ}$ follows from \cref{it_first_part_upper_link_proof} in the 
	proof of \cref{lem:restrictedtitsbuildingcmproperty} because $\ll \vec 
	e_{m+n} \rr_{\mbQ} \not\subseteq \ll \vec e_1, \dots, \vec e_m \rr_{\mbQ}$.
\end{proof}

\begin{definition} \label{bandba}
  Given an isotropic subspace $V \in T^{\omega, m}_n$, we obtain an isotropic summand $V \cap \mbZ^{2(m+n)}$ of $\mbZ^{2(m+n)}$ (see e.g.\ \cite[Lemma 2.4]{churchputman2017thecodimensiononecohomologyofslnz}) properly containing $\ll \vec e_1, \dots, \vec e_m \rr_\mbZ$.
  \begin{enumerate}
  \item Let $\mathcal{B}^m(V \cap \mbZ^{2(m+n)})$ be the full subcomplex of 
  $\Irel$ on the vertices satisfying $v \subseteq V \cap \mbZ^{2(m+n)}$.
  \item Let $\mathcal{BA}^m(V \cap \mbZ^{2(m+n)})$ be the full subcomplex of 
  $\Idelrel$ on the vertices satisfying $v \subseteq V \cap \mbZ^{2(m+n)}$.
  \end{enumerate}
\end{definition}

Let $\mbZ^{m+n} = \ll e_1, \dots, e_{m+n} \rr_{\mbQ} \cap \mbZ^{2(m+n)}$. 
By results of Maazen \cite{maazen1979thesis} and Church--Putman \cite{churchputman2017thecodimensiononecohomologyofslnz},
the complex $\mathcal{B}^m(\mbZ^{m+n})$ is Cohen--Macaulay of dimension $(n-1)$. The 
connectivity properties of $\mathcal{BA}^m(\mbZ^{m+n})$ have also been studied by Church--Putman \cite{churchputman2017thecodimensiononecohomologyofslnz}. We summarize the results contained in \cite{churchputman2017thecodimensiononecohomologyofslnz} in the following theorem.

\begin{theorem}[{\cite[Theorem 4.2 and Theorem C]{churchputman2017thecodimensiononecohomologyofslnz}}] \label{bbaconnectivity} Let $m \geq 0$ and $n \geq 1$.
  \begin{enumerate}
  \item \label{it_connectivity_Bm} $\mathcal{B}^m(\mbZ^{m+n})$ is $(n-2)$-connected and Cohen--Macaulay of dimension $(n-1)$.
  \item $\mathcal{BA}^m(\mbZ^{m+n})$ is $(n-1)$-connected.
  \end{enumerate}
\end{theorem}

The following two results allow us to relate the complexes introduced above. 
The first is a result of Quillen \cite{quillen1978homotopypropertiesofposets}. 
For this, we recall that the height $h(y)$ of an element $y$ in a poset $P$ is 
the length $l$ of the longest chain of the form $y_0 < \dots < y_l = y$ in $P$. 
If no such $l$ exists, we put $h(y) = \infty$.

\begin{theorem}[{\cite[Corollary 9.7]{quillen1978homotopypropertiesofposets}}]
	\label{quillen}
	Let $f: X \to Y$ be a poset map which is strictly increasing (if $x < x'$, 
	then $f(x) < f(x')$). Assume that $Y$ is Cohen--Macaulay of dimension $n$ 
	and that the poset fibers $f_{\leq y} = \{x \in X : f(x) \leq y\}$ are 
	Cohen--Macaulay of dimension $h(y)$ for all $y \in Y$. Then $X$ is 
	Cohen--Macaulay of dimension $n$.
\end{theorem}

The second result is a generalization of Quillen's \cite[Theorem 
9.1]{quillen1978homotopypropertiesofposets} due to van der Kallen--Looijenga.

\begin{theorem}[{\cite[Corollary 2.2]{vanderkallenlooijenga2011sphericalcomplexesattachedtosymplecticlattices}}] \label{vanderkallenlooijenga}
  Let $f: X \to Y$ be a poset map, $\theta \in \mbZ$, and $t: Y \to \mbZ$ an 
  increasing (if $y' < y$, then $t(y') < t(y)$) but bounded function. Suppose 
  that for every $y \in Y$, the poset fiber $f_{\leq y} = \{x \in X: f(x) \leq 
  y\}$ is $(t(y)-2)$-connected and that the upper link $Y_{>y}$ is $(\theta - 
  t(y) - 1)$-connected. Then the map $f$ is $\theta$-connected.
\end{theorem}

We are now ready to formulate our alternative arguments for the first items of 
\cref{isigmadelta-connectivity}. These arguments are completely formal and analogous to the proof of e.g.\ \cite[Proposition 1.2.]{vanderkallenlooijenga2011sphericalcomplexesattachedtosymplecticlattices}; the major difference being the input from \cref{restrictedtitsbuildingcm} and \cref{bbaconnectivity}.

\begin{lemma}
	\label{IrelW-connectivity}
	Let $m \geq 0$ and $n \geq 1$. The complex $\Irel(W)$ is Cohen--Macaulay of 
	dimension $n-1$. In particular, \cref{it_conn_IrelW} of 
	\cref{isigmadelta-connectivity} holds.
\end{lemma}

\begin{lemma}
	\label{I-connectivity}
	Let $m \geq 0$ and $n \geq 1$. The complex $\Irel$ is Cohen--Macaulay of dimension 
	$n-1$. In particular, \cref{it_conn_Irel} of 	
	\cref{isigmadelta-connectivity} holds.
\end{lemma}

\cref{I-connectivity} can easily be deduced from \cite[Proposition 6.13.2]{putman2009aninfinitepresentationofthetorelligroup}, i.e.\ \cref{it_conn_Irel} of \cref{isigmadelta-connectivity}, or \cite[Proposition 1.2]{vanderkallenlooijenga2011sphericalcomplexesattachedtosymplecticlattices}. Since it is used in the proof that $\IArel$ is highly connected at the end of this section and in the sequel \cite{brueckpatztsroka2023inpreparation}, we included a short argument.

\begin{proof}[Proof of \cref{IrelW-connectivity} and \cref{I-connectivity}]
Let $P(\Irel(W))$ and $P(\Irel)$ denote the simplex posets of $\Irel(W)$ and $\Irel$, respectively. The two lemmas follow by considering the poset maps
$$f: P(\Irel(W)) \to T^{\omega,m}_n(W): \Delta \mapsto \ll \vec e_1, \dots, 
\vec e_m \rr_{\mbQ} \oplus \ll \Delta \rr_\mbQ$$
and 
$$f: P(\Irel) \to T^{\omega,m}_n: \Delta \mapsto \ll \vec e_1, \dots, \vec 
e_m \rr_{\mbQ} \oplus \ll \Delta \rr_\mbQ,$$
respectively, and invoking \cref{quillen}. Let $V \in T^{\omega,m}_n(W)$ or $V \in T^{\omega,m}_n$. The application of 
Quillen's result relies on the facts that $h(V) = \dim(V) - m - 1$, that 
$T^{\omega,m}_n(W)$ and $T^{\omega,m}_n$ are Cohen--Macaulay posets of 
dimension $(n-1)$ (see \cref{cor:properties-of-restricted-tits-building}) and 
the observation that $f_{\leq V} = \mathcal{B}^m(V \cap \mbZ^{2(m+n)})$, which 
is Cohen--Macaulay of dimension $(\dim(V)-m-1)$ by \cref{bbaconnectivity}.
\end{proof}

Finally, we formulate an alternative argument for 
\cref{it_conn_IdelrelW} of \cref{isigmadelta-connectivity}. 

\begin{proof}[Proof of 	\cref{it_conn_IdelrelW} of 
\cref{isigmadelta-connectivity}]
We will show that $P(\Idelrel(W))$, the simplex poset of $\Idelrel(W)$, is $(n-1)$-connected. There is a poset 
map $$f: P(\Idelrel(W)) \to T^{\omega, m}_n(W): \Delta \mapsto \ll \vec e_1, 
\dots, \vec e_m \rr_{\mbQ} + \ll \Delta \rr_\mbQ.$$
Let $\theta = n$ and 
define $t: T^{\omega, m}_n(W) \to \mbZ: V \mapsto \dim(V)- m +1$. By 
\cref{cor:properties-of-restricted-tits-building}, we know that $T^{\omega, 
m}_n(W)_{>V}$ is $((n - 1) - (\dim(V) -m - 1) - 2) = (\theta - t(V) 
-1)$-connected. Furthermore, $f_{\leq V} = P(\mathcal{BA}^m(V \cap 
\mbZ^{2(m+n)}))$ is $(\dim(V) - m - 1) = (t(V) - 2)$-connected by 
\cref{bbaconnectivity} for $\dim(V) \geq 1 + m$.
Therefore \cref{vanderkallenlooijenga} implies that $f$ 
is $n$-connected. By \cref{cor:properties-of-restricted-tits-building}, the 
target is contractible, hence $\Idelrel(W)$ is $(n-1)$-connected.
\end{proof}

We end this section by explaining how Putman's connectivity result for 
$\Isigdelrel$ (see \cref{isigmadelta-connectivity}) implies high-connectivity 
of $\IArel$.

\begin{corollary}\label{iaconnectivity}
  If $m \geq 0$ and $n \geq 1$, then $\IArel$ is $(n-1)$-connected.
\end{corollary}

\begin{proof}
	By \cref{it_conn_isigdelrel} of \cref{isigmadelta-connectivity}, 
	the subcomplex $X_0 = \Isigdelrel$ of $X_1 = \IArel$ is $(n-1)$-connected. 
	We will apply the standard link argument explained in \cite[§2.1]{hatchervogtmann2017tethers} and \cite[Corollary 
	2.2]{hatchervogtmann2017tethers} to conclude that $X_1$ is $(n-1)$-connected as well.
	Let $B$ be the set of \emph{minimal} mixed simplices contained in $X_1$, 
	i.e.\ $B$ is the set of simplices $\Delta = \Delta' \ast \Theta$ in $\IArel$ consisting of a 
	$\sigma$ edge $\Theta = \{v, w\}$ and a 2-additive simplex of the form 
	$\Delta' = \{\ll \pm \vec v_1 \pm \vec v_2 \rr, v_1, v_2\}$ or $\Delta' = 
	\{\ll \pm \vec e_i \pm \vec v_1 \rr, v_1\}$ where $e_i \in \{e_1, \dots, 
	e_m\}$. Here, we call $\Delta' = \{\ll \pm \vec e_i \pm \vec v_1 \rr, 
	v_1\}$ a 2-additive simplex in $\IArel$ if $\{\ll \pm \vec e_i \pm \vec 
	v_1 \rr, v_1, e_i\}$ is a $2$-additive simplex in $\IA[m+n]$.
	We note that, by \cref{def_linkhat}, a simplex $\Delta$ in $\IArel$ is mixed (i.e.\ $\Delta \cup \{e_1, \dots, e_n\}$ is mixed in $\IA[m+n]$) 
	if and only if $\Delta$ has a unique face contained in $B$.
	This property implies that $B$ is a set of bad simplices in the sense of \cite[§2.1]{hatchervogtmann2017tethers}: 
	\cite[§2.1, Condition (1)]{hatchervogtmann2017tethers} holds, since any simplex in $X_1 = \IArel$ with no face in $B$ has to be in $X_0 = \Isigdelrel$.
	\cite[§2.1, Condition (2)]{hatchervogtmann2017tethers} holds, because if two faces of a simplex in $X_1 = \IArel$ are in $B$ these faces need to be equal.
	For $\Delta \in B$, the complex of simplices that are good for $\Delta$ is therefore 
	given by $\Link_{X_1}^{good}(\Delta) = \Link_{X_1}(\Delta)$.
	Now, $\Link_{X_1}(\Delta) \cong \Irel[n-\dim(\Delta)+1][m+\dim(\Delta)-2]$ and, by 
	\cref{I-connectivity}, this is (even more than) $(n-\dim(\Delta)-2)$-connected for every 
	minimal mixed simplex $\Delta$. Since $X_0 = \I[n]^{\sigma, \delta, m}$ is 
	$(n-1)$-connected (see \cref{isigmadelta-connectivity}), \cite[Corollary 
	2.2]{hatchervogtmann2017tethers} therefore implies that $X_1 = \IArel$ is 
	$(n-1)$-connected as well.
\end{proof}

\section{Symplectic integral apartment classes}
\label{sec:symplectic-modular-symbols}

Following \cite[Section 3]{gunnells2000symplecticmodularsymbols}, we explain 
the construction of the symplectic integral apartment class map appearing in 
\cref{gunnellstheorem}, $$[-]: \mbZ[\Sp{2n}{\mbZ}] \to \St^\omega_n = 
\tilde{H}_{n-1}(T^\omega_n; \mbZ).$$
The image $$[M] \in \St^\omega_n = \tilde{H}_{n-1}(T^\omega_n; \mbZ)$$ of an 
integral symplectic matrix $M \in \Sp{2n}{\mbZ}$ under this map is called its 
\emph{integral apartment class}. To define these homology classes, we use the 
following notation and observations.

\begin{definition} \label{apartmentmodel}
	Let $\sym{n} := \{ 1, \bar{1}, \dots, n, \bar{n}\}$. A nonempty subset $I \subseteq \sym{n}$ is called a standard subset if for all $1 \leq a \leq n$ it holds that $\{a,\bar{a}\} \not\subset I$. We denote by $\partial\beta_n$ the simplicial complex whose vertex set is $\sym{n}$ and whose $k$-simplices are the standard subsets $I \subset \sym{n}$ of size $k+1$.
\end{definition}

Observe that $\partial\beta_1 = \{1, \bar{1}\} \cong S^0$ and that the inclusion of vertex sets $\sym{n} \subseteq \sym{n+1}$ induces an inclusion of simplicial complexes $\partial\beta_n \hookrightarrow \partial\beta_{n+1}$ for any $n \in \mbN$. It is readily verified that $\partial\beta_{n+1}$ is exactly the simplicial join $\partial\beta_{n} \ast \{n+1, \overbar{n+1}\}$. It follows that $\partial\beta_n \cong \ast_{1}^{n} S^0$ is a simplicial sphere of dimension $n-1$ and that $\partial\beta_{n+1} = \partial\beta_{n} \ast \{n+1, \overbar{n+1}\}$ is obtained from $\partial\beta_n$ by suspension. We fix a fundamental class $\xi = \xi_{0} \in \tilde{H}_{0}(\partial\beta_1; \mbZ)$ once and for all. Using the suspension isomorphism, this class gives rise to fundamental classes $\xi = \xi_{n-1} \in \tilde{H}_{n-1}(\partial\beta_n;\mbZ)$ for all $n \in \mbN$.

Given an integral symplectic matrix $M \in \Sp{2n}{\mbZ}$, its 
column vectors form a symplectic basis of $\mbQ^n$. We may index the $2n$ 
column vectors from left to right by $\sym{n} =  \{ 1, \bar{1}, \dots, n, 
\bar{n}\}$. By the definition of the symplectic form $\omega$, this indexing $M 
= (\vec M_a)_{a \in \sym{n}}$ has the property that if $I \in \partial\beta_n$ 
is a simplex, then $$M_I = \langle \{\vec M_a: a \in I\} \rangle_\mbQ$$ is an 
isotropic 
subspace of $\mbQ^{2n}$. This implies 
that for every $M \in \Sp{2n}{\mbZ}$, we can define a poset map
$$ \partial M : P(\partial\beta_n) \to T_n^\omega: I \mapsto M_I,$$
where $P(\partial\beta_n)$ denotes the poset of simplices of $\partial\beta_n$. 
Note that, passing to the associated order complexes of these posets, $\partial 
M$ defines a simplicial embedding. Since the order complex of the poset 
$P(\partial\beta_n)$ 
is the barycentric subdivision of $\partial\beta_n$, it follows that the image 
of the map $\partial M$ is a subcomplex that is homeomorphic to an 
$(n-1)$-sphere. Such a subcomplex is called an \emph{integral apartment} of the 
Tits building $T_n^\omega$. Taking homology, we obtain a map
$$\partial M_{\star} : \tilde{H}_{n-1}(P(\partial\beta_n); \mbZ) \to \St^\omega_n.$$

Barycentric subdivision of simplicial complexes comes with a natural homology isomorphism on chain level $b: C_\star \to C_\star \circ P$, where $C_\star$ assigns an ordered simplicial complex its simplicial chain complex with trivial $\mbZ$-coefficients. Using the induced isomorphism
$$b : \tilde{H}_{n-1}(\partial\beta_n;\mbZ) \to \tilde{H}_{n-1}(P(\partial\beta_n);\mbZ),$$
we obtain a unique class $b(\xi) \in \tilde{H}_{n-1}(P(\partial\beta_n);\mbZ)$ for every $n \in \mbN$.

\begin{definition} \label{modularsymbolmap}
	The \emph{symplectic integral apartment class} $[M] \in \St^\omega_n$ of $M 
	\in \Sp{2n}{\mbZ}$ is defined to be the value of $\partial M_{\star} : 
	\tilde{H}_{n-1}(P(\partial\beta_n);\mbZ) \to \St^\omega_n$ at $b(\xi)$,
	$$[M] := \partial M_{\star}(b(\xi)).$$
	This defines a map $$[-]: \mbZ[\Sp{2n}{\mbZ}] \to \St^\omega_n: M \mapsto 
	[M]$$ which we called the \emph{symplectic integral apartment class map}.
\end{definition}
\begin{remark} \label{rationalsymbols}
	The construction of symplectic apartment classes described above also works 
	if one starts with an element in the rational symplectic group $M \in 
	\Sp{2n}{\mbQ}$. This leads to the definition of a symplectic 
	\emph{rational} apartment class map
	$$[-]: \mbZ[\Sp{2n}{\mbQ}] \to \St^\omega_n.$$
	It follows from the proof of the Solomon--Tits Theorem (see 
	\cite{solomon1969thesteinbergcharacterofafinitegroup} or \cite[Section IV.5, Theorem 2]{brown1989buildings}) 
	that this map is a surjection.
	Gunnells' theorem states that its restriction to the ``much smaller'' group ring 
	$\mbZ[\Sp{2n}{\mbZ}]$ is still a surjection.
	The ``smallness'' of $\mbZ[\Sp{2n}{\mbZ}]$ can be illustrated from a building theoretic perspective: Every apartment in the complete system of apartments of the building $T_n^\omega$ (in the sense of \cite[Section IV.4]{brown1989buildings}) can be obtained by a $\Sp{2n}{\mbQ}$-translation, but not by a $\Sp{2n}{\mbZ}$-translation, of the standard apartment.
\end{remark}
\begin{remark}
	Gunnells' proof for \cref{gunnellstheorem} was inspired by work of 
	Ash--Rudolph 
	\cite{ashrudolph1979themodularsymbolandcontinuedfractionsinhigherdimensions}
	 and based on the content of \cref{rationalsymbols}. The general strategy 
	is to devise an algorithm that takes as input a \emph{rational} apartment 
	classes $[M] \in \St^\omega_n$ for $M \in \Sp{2n}{\mbQ}$ and outputs a 
	linear combination of \emph{integral} apartment classes that is equal to 
	$[M]$.
\end{remark}
\begin{remark}
	The integral apartment classes introduced in 
	\cref{modularsymbolmap} are called ``unimodular symbols'' in 
	\cite{gunnells2000symplecticmodularsymbols}, and 
	the rational apartment classes in \cref{rationalsymbols} are 
	called ``symplectic modular symbols'' in 
	\cite{gunnells2000symplecticmodularsymbols}.
	More generally, the terminology that we use in this note is close to that of \cite{churchfarbputman2019integralityinthesteinbergmodule, 	churchputman2017thecodimensiononecohomologyofslnz, 	brueckmillerpatztsrokawilson2022onthecodimensiontwocohomologyofslnz,	bruecksantosregosroka2022onthetopdimensionalcohomologyofarithmeticchevalleygroups}, while the terminology in Gunnells' paper 	\cite{gunnells2000symplecticmodularsymbols} is close to that of 	\cite{ashrudolph1979themodularsymbolandcontinuedfractionsinhigherdimensions}.
\end{remark}

\section{A new proof of Gunnells' theorem}
\label{sec:new-proof-of-gunnells-theorem}
In this final section, we present a new proof of Gunnells' \cref{gunnellstheorem}, i.e.~we prove that the integral apartment class map $$[-]: \mbZ[\Sp{2n}{\mbZ}] \to \St^{\omega}_n$$ is surjective. 
Our strategy is to factor it into a composition of four maps and then verify that each of these is a surjection. This is analogous to the strategy employed by Church--Farb--Putman in \cite{churchfarbputman2019integralityinthesteinbergmodule}. Gunnells' theorem then follows from the following two propositions, whose proof we will explain in the remainder of this work.

\begin{proposition} \label{firststep} If $n \geq 1$, there exists a commutative diagram of the following shape
  \begin{center}
    \begin{tikzcd}
      \mbZ[\Sp{2n}{\mbZ}] \arrow[d, "\alpha"] \arrow[ddrr, bend left, "{[-]}"] & &\\
      H_n(\IA, \I^\delta) \arrow[r, "\delta"] & \tilde{H}_{n-1}(\I^\delta) 
      \arrow[d, "b"] &\\
      & \tilde{H}_{n-1}(P(\I^\delta)) \arrow[r, "s_\star"] & 
      \tilde{H}_{n-1}(T_n^\omega) = \St^{\omega}_n
    \end{tikzcd}
  \end{center}
  where $[-]: \mbZ[\Sp{2n}{\mbZ}] \to \St^{\omega}_n$ is the integral apartment class map, $\delta$ is the connecting morphism of the long exact sequence of the pair $(\IA, \I^\delta)$ and $b$ is the homology isomorphism coming from barycentric subdivision.
\end{proposition}

The morphisms $\alpha$ and $s_\star$ in the statement of \cref{firststep} are 
defined below.

\begin{proposition} \label{secondstep} If $n \geq 1$, then the maps occurring in \cref{firststep} satisfy:
  \begin{enumerate}
  \item \label{secondstep_one} $\alpha: \mbZ[\Sp{2n}{\mbZ}] \to H_n(\IA, 
  \I^\delta)$ is a surjection.
  \item \label{secondstep_two} $\delta: H_n(\IA, \I^\delta) \to 
  \tilde{H}_{n-1}(\I^\delta)$ is a 
  surjection.
  \item \label{secondstep_three} $b : \tilde{H}_{n-1}(\I^\delta) \to 
  \tilde{H}_{n-1}(P(\I^\delta))$ is an isomorphism.
  \item \label{secondstep_four} $s_\star: \tilde{H}_{n-1}(P(\I^\delta)) \to 
  \St^{\omega}_n$ is an  isomorphism.
  \end{enumerate}
\end{proposition}

To define the morphisms $\alpha$ and $s_\star$ in the 
statement of \cref{firststep}, we start by introducing a simplicial complex, 
which is closely related to the complex $\partial\beta_n$ occurring in the 
definition of the apartment class map.

\begin{definition}
  We call a nonempty subset $I \subset \sym{n}$ a \emph{$\sigma$ subset}, if $\{n, \bar{n}\} \subset I$ and for all $1 \leq a \leq n-1: \{a, \bar{a}\} \not\subset I$. Let $\beta_n$ be the simplicial complex with vertex set $\sym{n}$ and $k$-simplices subsets $I \subset \sym{n}$ of size $k+1$, which are either standard (see \cref{apartmentmodel}) or $\sigma$ subsets.
\end{definition}

Note that $\beta_1 \cong D^1$. Furthermore, $\beta_n \cong (\ast^{n-1}_1 S^0) \ast D^1$ is homeomorphic to a disc of dimension $n$ whose boundary sphere is triangulated by the subcomplex $\partial\beta_n \subset \beta_n$, i.e.\ $$(|\beta_n|, |\partial\beta_n|) \cong (D^n, S^{n-1})$$

The definition of the map $\alpha$ involves the following construction: Let $M 
= (\vec M_a)_{a \in \sym{n}} \in \Sp{2n}{\mbZ}$. Given a $k$-simplex $I$ of 
$\beta_n$, we find an associated simplex $M_I^\alpha = \{\langle \vec M_a 
\rangle_{\mbZ} : a \in I \}$ of $\IA$: If $I$ is a standard subset, then 
$M_I^\alpha$ 
is a standard simplex. If $I$ is $\sigma$ subset, then $M_I^\alpha$ is a 
$\sigma$ simplex. The resulting map$$M^\alpha: \beta_n \to \IA$$is a simplicial 
embedding and the boundary $$\partial M^\alpha: \partial\beta_n \to \IA$$ of 
this simplicial disc is contained in $\Idel$, i.e.\ $$\partial M^\alpha: 
\partial\beta_n \to \I^\delta \to \IA.$$

\begin{definition}
	The value $\alpha(M)$ of the map $\alpha: \mbZ[\Sp{2n}{\mbZ}] \to H_n(\IA, 
	\I^\delta)$ at a matrix $M = (\vec M_a)_{a \in \sym{n}} \in \Sp{2n}{\mbZ}$ 
	is 
	defined to be the image of the fundamental class $\xi \in 
	\tilde{H}_{n-1}(\partial\beta_n)$ under the composition
\begin{center}
  \begin{tikzcd}
    \tilde{H}_{n-1}(\partial\beta_n) & H_{n}(\beta_n, \partial\beta_n) 
    \arrow[l, "\cong", swap] \arrow{rrr}{(M^\alpha, \partial M^\alpha)_\star} 
    &&& H_n(\IA, \I^\delta),
  \end{tikzcd}
\end{center}
where the first isomorphism is the connecting morphism associated to the pair 
$(\beta_n, \partial\beta_n)$.
\end{definition}

Finally, we define the map $s_\star$.

\begin{definition} 
	$s_\star: \tilde{H}_{n-1}(P(\I^\delta)) \to \St^{\omega}_n$ is the map 
	induced in homology by the spanning map 
	\[
	s: P(\I^\delta) \to T_n^\omega: \Delta \mapsto \langle \Delta \rangle_\mbQ,
	\]
	where $P(\I^\delta)$ denotes the poset of simplices of $\I^\delta$.
\end{definition}

\subsection{Proof of \texorpdfstring{\cref{firststep}}{Proposition 5.1}} Let $M 
\in \Sp{2n}{\mbZ}$. We need to verify that $$[M] = (s_\star \circ b \circ 
\delta \circ \alpha)(M).$$

Consider the following diagram:

\begin{center}
  \begin{tikzcd}
    H_{n}(\beta_n, \partial\beta_n) \arrow[r, "\cong"] \arrow{d}{(M^\alpha, 
    \partial M^\alpha)_\star} & \tilde{H}_{n-1}(\partial\beta_n) 
    \arrow{d}{\partial M_\star^\alpha} \arrow[r, "b"] & 
    \tilde{H}_{n-1}(P(\partial \beta_n)) \arrow{d}{P(\partial M^\alpha)_\star}\\
    H_n(\IA, \I^\delta) \arrow[r, "\delta"] & \tilde{H}_{n-1}(\I^\delta) 
    \arrow[r, "b"] & \tilde{H}_{n-1}(P(\I^\delta))
  \end{tikzcd}
\end{center}

The left square commutes because the connecting morphism of the long exact sequence of a pair is a natural transformation. The right square commutes because $b: C_\star \to C_\star \circ P$ is a natural homology isomorphism. It follows that
$$(b \circ \delta \circ \alpha)(M) = (P(\partial M^\alpha)_\star \circ b)(\xi) 
\in \tilde{H}_{n-1}(P(\I^\delta)).$$
To complete the proof, we need to see that
$$(s \circ P(\partial M^\alpha) \circ b)_\star(\xi) = [M],$$
where $[M] = (\partial M \circ b)_\star (\xi)$ is as in \cref{modularsymbolmap}. This holds because the composition $(s \circ P(\partial M^\alpha))$ defined in this section is equal to the map $\partial M$ defined in the paragraph before \cref{modularsymbolmap}.

\subsection{Proof of \texorpdfstring{\cref{secondstep}}{Proposition 5.2}} The 
arguments for \cref{secondstep_two}, \cref{secondstep_three} and 
\cref{secondstep_four} of \cref{secondstep} are similar to the arguments used 
by Church--Farb--Putman in the setting of $\SL{n}{\mbZ}$
\cite{churchfarbputman2019integralityinthesteinbergmodule}. However, while the 
analogue of the surjectivity of the map $\alpha_n: 
\mbZ[\Sp{2n}{\mbZ}] \to H_n(\IA, \I^\delta)$ is rather immediate for special 
linear groups, this step (i.e.\ \cref{secondstep_one} of \cref{secondstep}) is 
more involved for symplectic groups. The reason is that 
apartments in the Tits building of type $\mathtt{A}_{n-1}$, which is used in 
the argument for $\SL{n}{\mbZ}$, have the same combinatorial structure as the 
boundary of an $(n-1)$-simplex $\partial \Delta^{n-1}$ and can therefore be 
``filled'' by gluing in a single simplex of dimension $n-1$. Apartments of the 
Tits building of type $\mathtt{C}_n$, which we use here for $\Sp{2n}{\mbZ}$, 
have a different simplicial structure. They are modelled by the complex 
$\partial \beta_n$ and require multiple simplices to be ``filled''. In the 
complex $\IA$, this is achieved by $\sigma$ simplices. Observe that $\sigma$ 
simplices already occur in dimension one. Therefore, and in contrast to the 
analogous situation for special linear groups (see \cite[Step 1, 2.3 Proof of 
Theorem B]{churchfarbputman2019integralityinthesteinbergmodule}), the relative 
chain complex $C_\star(\IA,\I^\delta)$ is nontrivial in degree $\star = n-1$.

\begin{proof}[Proof of \cref{secondstep_two} of \cref{secondstep}] 
\cref{iaconnectivity} implies that $\tilde{H}_{n-1}(\IA) = 0$.  Hence, the long 
exact sequence of the pair $(\IA, \I^\delta)$ implies that $\delta: 
H_n(\IA, \I^\delta) \to \tilde{H}_{n-1}(\I^\delta)$ is surjective.
\end{proof}

\begin{proof}[Proof of \cref{secondstep_three} of \cref{secondstep}] 
The map $b : \tilde{H}_{n-1}(\I^\delta) \to \tilde{H}_{n-1}(P(\I^\delta))$ is 
an isomorphism by definition. It is induced by a natural homology isomorphism 
of chain complexes.
\end{proof}

\begin{proof}[Proof of \cref{secondstep_four} of \cref{secondstep}] 
To verify that the map $s_\star: \tilde{H}_{n-1}(P(\I^\delta)) \to 
\tilde{H}_{n-1}(T^\omega_n) = \St^{\omega}_n$ is an isomorphism, we can apply 
\cref{vanderkallenlooijenga} once more. Let $\theta = n$. Let $V \in 
T^\omega_n$ and set $t(V) = \dim(V) + 1$. 
\cref{cor:properties-of-restricted-tits-building} for $m = 0$ implies that the 
upper link $(T^{\omega}_n)_{>V}$ is $((n - 1) - (\dim(V) - 1) - 2) = (\theta - 
t(V) -1)$-connected. \cref{bbaconnectivity} for $m = 0$ implies that the lower 
fiber $f_{\leq V} = \mathcal{BA}(V \cap \mbZ^n)$ is $(\dim(V)-1) = 
(t(V)-2)$-connected. Therefore, it follows from \cref{vanderkallenlooijenga} 
that $s: 
P(\I^\delta) \to T^\omega_n$ is $n$-connected and, hence, that the map 
$s_\star: \tilde{H}_{n-1}(P(\I^\delta)) \to \tilde{H}_{n-1}(T^\omega_n) = 
\St^{\omega}_n$ is an 
isomorphism.
\end{proof}

Let $n \in \mbN$. In the following, $E^\omega_n$ denotes the set of all 
$\sigma$ edges in $\IA$. The proof of \cref{secondstep_one} of 
\cref{secondstep} is by induction on $n \geq 1$. The base case $n = 1$ is 
the content of the next lemma.

\begin{lemma}
	\label{secondstep_one_ib}
	If $n=1$, then $\alpha_n: \mbZ[\Sp{2n}{\mbZ}] \to H_n(\IA, \I^\delta)$ is 
	surjective.
\end{lemma}

\begin{proof}
	For $2n = 2$, it follows that $\IA[1]$ is a one-dimensional connected 
	simplicial complex, that all edges are $\sigma$ edges and that 
	$\I[1]^\delta = \I[1]$ is exactly the $0$-skeleton of $\IA[1]$.\footnote{In 
	fact, $\Sp{2}{\mbZ} = \SL{2}{\mbZ}$ and $\IA[1]$ is isomorphic to the 
	1-dimensional complex of partial frames $\mathcal{B}(\mbZ^2)$. The complex 
	$\mathcal{B}(\mbZ^2)$ is discussed in detail in the introduction of 
	\cite{churchputman2017thecodimensiononecohomologyofslnz} (see paragraph 
	``Improving connectivity: the complex of partial augmented frames'').} In 
	particular, $$H_1(\IA[1], \I[1]^\delta) \cong \bigoplus_{\Delta \in 
	E^\omega_1} \mbZ.$$
	Given some $M \in \Sp{2}{\mbZ}$ with $M = (\vec v,\vec w)$, we see that 
	$M^\alpha(\beta_1) \subset \IA[1]$ is exactly the $\sigma$ edge $\Delta = 
	\{ v, w\}$ and $M^\alpha(\partial\beta_1) \subset \I[1]^\delta$ is exactly 
	the boundary of this edge. Hence, under the identification above, 
	$\alpha_1$ maps the symplectic matrix $M$ to a generator of the 
	$\mbZ$-summand indexed by $\Delta = \{ v,  w\}$. Given any $\sigma$ edge 
	$\Delta = \{ v,  w\}$, we have that $\omega(\vec v, \vec w) = \pm 1$. Thus, 
	for some choice of signs $(\pm \vec v,\pm \vec w) \in \Sp{2}{\mbZ}$. It 
	follows that $\alpha_n$ is surjective for $n = 1$.
\end{proof}

\begin{proof}[Proof of \cref{secondstep_one} of \cref{secondstep}]
To see that the map $$\alpha_n: \mbZ[\Sp{2n}{\mbZ}] \to H_n(\IA, \I^\delta)$$ 
is surjective, we perform an induction on $n \geq 1$. The induction beginning 
$n = 1$ is \cref{secondstep_one_ib}. Let $n > 1$ and assume that 
\cref{secondstep_one} of \cref{secondstep} holds for $1 \leq k \leq n-1$. We 
deduce the surjectivity of the map $\alpha_n$ in two steps.

The first step is to show that the target $H_n(\IA, \I^\delta)$ is 
a direct sum of ``smaller'' Steinberg modules. For this, we observe that $\IA$ 
is obtained from $\I^\delta$ via the following pushout diagram:
\begin{center}
  \begin{tikzcd}
    \bigsqcup_{\Delta \in E^\omega_n} \Sigma^1 \Link_{\IA}(\Delta) \arrow[r, "\cong"] &\bigsqcup_{\Delta \in E^\omega_n} \Star_{\IA}(\Delta) \cap \I^\delta \arrow[r] \arrow[d, hook] & \I^\delta \arrow[d, hook]\\
    &\bigsqcup_{\Delta \in E^\omega_n} \Star_{\IA}(\Delta) \arrow[r] & \IA
  \end{tikzcd}
\end{center}

In particular, excision implies that 
\begin{equation} \label{eq1}
  H_n(\IA, \I^\delta) \cong \bigoplus_{\Delta \in E^\omega_n} H_n(\Star_{\IA}(\Delta), \Sigma^1 \Link_{\IA}(\Delta)).
\end{equation}

The contractibility of $\Star_{\IA}(\Delta)$ implies that the connecting morphism of the pair $$(\Star_{\IA}(\Delta), \Sigma^1 \Link_{\IA}(\Delta))$$ is an isomorphism
\begin{equation} \label{eq2}
  H_n(\Star_{\IA}(\Delta), \Sigma^1 \Link_{\IA}(\Delta)) \xrightarrow{\delta_n} 
  \tilde{H}_{n-1}(\Sigma^1 \Link_{\IA}(\Delta)).
\end{equation}

The suspension isomorphism gives an identification
\begin{equation} \label{eq3}
  \tilde{H}_{n-1}(\Sigma^1 \Link_{\IA}(\Delta)) \xrightarrow{\Sigma^{-1}} 
  \tilde{H}_{n-2}(\Link_{\IA}(\Delta)).
\end{equation}

Observe that $\Link_{\IA}(\Delta) = \mathcal{I}^\delta(\Delta^\perp)$, where $\Delta^\perp \coloneqq \langle \Delta \rangle^\perp \subset 
\mbQ^{2n}$. We therefore proved that
\begin{equation} \label{eq4}
  H_n(\IA, \I^\delta) \cong \bigoplus_{\Delta \in E^\omega_n} 
  \tilde{H}_{n-2}(\mathcal{I}^\delta(\Delta^\perp)) \cong \bigoplus_{\Delta \in 
  E^\omega_n} \St^\omega(\Delta^\perp)
\end{equation}
where the last isomorphism is obtained by invoking \cref{secondstep_three} and 
\cref{secondstep_four} of \cref{secondstep} and $\St^\omega(\Delta^\perp)$ 
denotes the Steinberg module of the symplectic subspace $\Delta^\perp \subset 
\mbQ^{2n}$. This completes the first step.

The second step of the proof that $\alpha_n: \mbZ[\Sp{2n}{\mbZ}] \to H_n(\IA, 
\I^\delta)$ is surjective is to decompose the domain $\mbZ[\Sp{2n}{\mbZ}]$ in a 
compatible way and identify the resulting map on each summand. This is the 
content of the following claim.

\begin{claim} 
  \label{identification}
  Let $\Delta = \{v, w\} \in E^\omega_n$ and let $\omega(\vec v, \vec w) = 1$, 
  $\tilde\Delta = (\vec v, \vec w)$ an ordered pair. Let 
  $\mbZ[\Sp{}{\tilde\Delta^\perp}] \subset \mbZ[\Sp{2n}{\mbZ}]$ be the 
  $\mbZ$-summand spanned by symplectic matrices $M \in \Sp{2n}{\mbZ}$ 
  satisfying $\vec M_{n} = \vec v$ and $\vec M_{\bar{n}} = \vec w$. The 
  sequence of 
  identifications above yields a map 
  $$[-]_{\tilde\Delta}:\mbZ[\Sp{}{\tilde\Delta^\perp}] \to 
  H_{n-1}(\mathcal{IA}(\Delta^\perp), \mathcal{I}^\delta(\Delta^\perp)) \to 
  \tilde{H}_{n-2}(\mathcal{I}^\delta(\Delta^\perp)) \to 
  \St^\omega(\Delta^\perp)$$ that is exactly the integral apartment class map 
  of the group $\Sp{}{\Delta^\perp}$ of symplectic automorphisms of the summand 
  $\Delta^\perp \subset \mbZ^{2n}$.
\end{claim}

Before proving this claim, we explain how this finishes the proof of the 
induction step and hence of \cref{secondstep}. \autoref{identification} implies 
that the following diagram commutes.

\begin{center}
  \begin{tikzcd}
    \mbZ[\Sp{2n}{\mbZ}] \arrow[r, equal] \arrow[d, "\alpha_n"] & 
    \bigoplus_{\tilde\Delta = (\vec M_{n},\vec M_{\bar{n}})} 
    \mbZ[\Sp{}{\tilde\Delta^\perp}] \arrow[d, "\oplus {[-]}_{\tilde\Delta}"]\\
    H_n(\IA, \I^\delta) \arrow[r, "\cong"] & \bigoplus_{\Delta \in E^\omega_n} 
    \St^\omega(\Delta^\perp)
  \end{tikzcd}
\end{center}

The induction hypothesis, \autoref{identification} and \cref{secondstep_two}, 
\cref{secondstep_three} and \cref{secondstep_four} of \cref{secondstep} imply 
that the integral apartment class maps occurring on the right hand side of the 
above diagram,
$$[-]_{\tilde\Delta}:\mbZ[\Sp{}{\tilde\Delta^\perp}] \to 
\St^\omega(\Delta^\perp),$$
are surjective. For any $\Delta \in E^\omega_n$, there exists an ordered pair 
$\tilde\Delta = (\vec M_{n},\vec M_{\bar{n}})$ such that $\Delta = \{\langle 
\vec M_{n} \rangle_{\mbZ}, \langle \vec M_{\bar{n}} \rangle_{\mbZ}\}$. It 
follows that the right vertical map in the diagram is surjective. Therefore, 
$\alpha_n$ is surjective as well.
\end{proof}

\begin{proof}[Proof of \autoref{identification}.] It suffices to consider the 
case where $(\vec v, \vec w) = (\vec e_{n}, \vec f_{n})$ consists of the last 
symplectic pair of the standard symplectic basis. All other cases can be 
reduced to this case by applying a symplectic matrix that sends $(\vec v, \vec 
w)$ to $(\vec e_{n}, \vec f_{n})$. Let 
$\Delta = \{ \langle \vec e_{n} \rangle_{\mbZ}, \langle \vec f_{n} 
\rangle_{\mbZ}\}$, $\tilde\Delta = (\vec e_{n}, \vec f_{n})$ and $M \in 
\Sp{2n}{\mbZ}$ a symplectic matrix with $\vec M_{n} = \vec e_{n}$ and $\vec 
M_{\bar{n}} = 
\vec f_{n}$. The symplectic relations imply that the $\vec e_{n}$- and $\vec 
f_{n}$-coordinates of all other column vectors $\vec M_a, \vec 
M_{\bar{a}}$, where $a \in \{1, \dots, n-1\}$, of $M$ are zero. In 
particular, $M$ corresponds to a unique 
element $\widetilde{M}$ of the symplectic group $\Sp{}{\Delta^\perp}$ of the 
summand $\Delta^\perp \subset \mbZ^{2n}$ and vice versa. Recall that the class 
$\alpha(M) \in H_n(\IA, \I^\delta)$ was defined using the map of pairs:
$$(M^\alpha, \partial M^\alpha): (\beta_n, \partial \beta_n) \to (\IA, \I^\delta)$$
This map factors through the pair
$$(\Star_{\IA}(\Delta), \Star_{\IA}(\Delta) \cap \I^\delta) \cong (\Star_{\IA}(\Delta), \Sigma^1 \Link_{\IA}(\Delta)).$$
The naturality of connecting morphisms yields a commutative diagram:
\begin{center}
	\adjustbox{scale=0.9,center}{
  	\begin{tikzcd}
    H_n(\beta_n, \partial \beta_n) \arrow[r, "\delta"] \arrow{d}{(M^\alpha, 
    \partial M^\alpha)_\star} & \tilde{H}_{n-1}(\partial\beta_n) 
    \arrow{d}{\partial M^\alpha_\star} \arrow[r, "\Sigma^{-1}"] & 
    \tilde{H}_{n-2}(\partial \beta_{n-1}) 
    \arrow{d}{\partial\widetilde{M}^\alpha_\star}\\
    H_n(\Star_{\IA}(\Delta), \Sigma^1 \Link_{\IA}(\Delta)) \arrow[r, "\delta"] 
    & \tilde{H}_{n-1}(\Sigma^1 \Link_{\IA}(\Delta)) \arrow[r, "\Sigma^{-1}"] & 
    \tilde{H}_{n-2}(\Link_{\IA}(\Delta))
  	\end{tikzcd}
	}
\end{center}

Hence, under the identifications in \cref{eq1}, \cref{eq2} and \cref{eq3}, the class $$\alpha(M) \in H_n(\IA, \I^\delta)$$ is mapped to
$$\partial\widetilde{M}^\alpha_\star(\xi_{n-2}) \in 
\tilde{H}_{n-2}(\Link_{\IA}(\Delta)) = 
\tilde{H}_{n-2}(\mathcal{I}^\delta(\Delta^\perp)).$$
The following commuting square proves that 
$\partial\widetilde{M}^\alpha_\star(\xi_{n-2})$ is exactly $(\delta \circ 
\alpha_{n-1})(\widetilde{M})$:

\begin{center}
  \begin{tikzcd}
    H_n(\beta_{n-1}, \partial\beta_{n-1}) \arrow[r, "\delta"] 
    \arrow{d}{(\widetilde{M}^\alpha, \partial \widetilde{M}^\alpha)_\star} & 
    \tilde{H}_{n-1}(\partial\beta_{n-1}) \arrow{d}{\partial 
    \widetilde{M}_\star^\alpha}\\
    H_n(\mathcal{IA}(\Delta^\perp), \mathcal{I}^\delta(\Delta^\perp)) \arrow[r, 
    "\delta"] & \tilde{H}_{n-1}(\mathcal{I}^\delta(\Delta^\perp))
  \end{tikzcd}
\end{center}

Hence, the final identification used in \cref{eq4} and \cref{firststep} yield that $\alpha_n(M)$ is mapped to 
\[
(s_\star \circ b \circ \delta \circ \alpha_{n-1})(\widetilde{M}) = [\widetilde{M}] \in \St^\omega(\Delta^\perp) \qedhere
\]
\end{proof}

\emergencystretch=2em
\printbibliography

\end{document}